\documentclass[12pt, reqno]{amsart}
\usepackage{amsmath,amssymb,amsthm}
\usepackage{amsmath, amssymb}
\usepackage{amsthm, amsfonts, mathrsfs}
\usepackage{mathptmx}
\usepackage{fullpage}
\usepackage{amsfonts,graphicx}
\numberwithin{equation}{section}
\usepackage[colorlinks=true, pdfstartview=FitV, linkcolor=blue, citecolor=blue, urlcolor=blue]{hyperref}

\newtheorem {theorem}{Theorem}[section]
\newtheorem {lemma}[theorem]{{\bf Lemma}}

\newtheorem {proposition}[theorem]{{\bf Proposition}}
\theoremstyle{remark}

\theoremstyle{plain} \numberwithin {equation}{section}
\def\na{\nabla}
\newcommand{\p}{\partial}
\def\R{\mathbb{R}}

\def\div{ \hbox{\rm div}\,  }
\def\u{ \mathbf{u} }
\def\b{ \mathbf{B} }

\def\f{\frac}
\def\la{ \nabla }
\def\nn{\nonumber}
\newcommand{\norm}[2]{\lVert #1 \rVert_{#2}}

\begin{document}

\title [Global solutions to 3D MHD equations]{Global solutions to 3D compressible MHD equations with partial\\[1ex] magnetic diffusion}

\author[J. Wu]{Jiahong Wu}
\address[J. Wu]{Department of Mathematics, University of Notre
	Dame, Notre Dame, IN 46556, USA } \email{jwu29@nd.edu}

\author[X. Zhai]{Xiaoping Zhai}
\address[X. Zhai]{School of Mathematics and Statistics, Guangdong University of Technology,
Gua-ngzhou, 510520, China} \email{pingxiaozhai@163.com (Corresponding author)}


\begin{abstract}
The global existence of strong solutions to the compressible viscous magnetohydrodynamic (MHD) equations in $\R^3$ remains a significant open problem. When there is no magnetic diffusion, even small data global well-posedness is unknown. This study investigates the Cauchy problem in $\R^3$ for the compressible viscous MHD equations with horizontal magnetic diffusion. Using various anisotropic Sobolev inequalities and sharp estimates, we establish the existence of global solutions under small initial data within the Sobolev space framework.
\end{abstract}
\maketitle

%\tableofcontents

\section{Introduction and main result}
This paper examines the Cauchy problem in $\R^3$ for the 3D compressible viscous magnetohydrodynamic (MHD) equations with partial dissipation in the magnetic field. The MHD system under study is given by
\begin{eqnarray}\label{m1}
	\left\{\begin{aligned}
		&\partial_t \rho + {\mathrm{div}} (\rho \mathbf{u}) =0,  \\
		&\rho\partial_t \mathbf{u}+ \rho\mathbf{u}\cdot\nabla \mathbf{u}-\mu\Delta \mathbf{u} - (\lambda+\mu) \nabla {\mathrm{div}} \mathbf{u}+\nabla P=(\nabla\times\mathbf{B})\times\mathbf{B},\\
		&\partial_t \mathbf{B}-\sigma\Delta_h \mathbf{B}+\mathbf{u}\cdot\nabla\mathbf{B}+\mathbf{B}\div\mathbf{u}=\mathbf{B}\cdot\nabla\mathbf{u},\\
		&{\mathrm{div}} \mathbf{B} =0,\\
&(\rho,\u,\mathbf{B})|_{t=0}=(\rho_0,\u_0,\mathbf{B}_0),
	\end{aligned}\right.
\end{eqnarray}
 defined for $(t, x) \in \mathbb{R}_+ \times \mathbb{R}^3$. Here, the variables $\rho = \rho(t, x)$, $\mathbf{u} = \mathbf{u}(t, x)$, and $\mathbf{B} = \mathbf{B}(t, x)$ represent the fluid density, velocity field, and magnetic field, respectively. The pressure $P = P(\rho)$ is a smooth function of density, satisfying the conditions $P'(\rho) > 0$ and $P'(\bar{\rho}) = 1$, where $\bar{\rho} > 0$ is a constant reference density. The parameters $\mu$ and $\lambda$ denote shear viscosity and volume viscosity coefficients, respectively, satisfying the standard strong parabolicity assumption.
\begin{align*}
	\mu>0\quad\hbox{and}\quad
	\nu\stackrel{\mathrm{def}}{=}\lambda+2\mu>0.
\end{align*}
$\sigma>0$ is the electrical conductivity of the magnetic field. For notational convenience, we have written
\begin{align*}
\Delta_h=\partial_1^2+\partial_2^2
\end{align*}
and we shall also write $\nabla_h=(\partial_1,\partial_2)$.

We shall consider the Cauchy
problem of \eqref{m1} in $\R^3$ with the initial data satisfying
\begin{align}\label{long}
(\rho_0(x), \u_0(x), \b_0(x))\to(\bar{\rho}, \mathbf{0}, \mathbf{0}) \quad \mathrm{ for }
\quad|x|\to \infty,
\end{align}
where  $\bar{\rho} $  is a    given  positive constant.

\vskip .1in
MHD systems provide a fundamental framework for describing the dynamics of electrically conducting fluids, such as plasmas, liquid metals, and electrolytes. In these systems, the velocity field obeys the Navier-Stokes equations, modified by the Lorentz force due to the magnetic field, while the magnetic field evolves according to Maxwell's equations. The MHD equations play a crucial role in understanding various phenomena across geophysics, astrophysics, cosmology, and engineering (see, e.g., \cite{Bis,Davi, Pri}). Below, we provide a brief overview of key results in this field.
\begin{itemize}
\item When $\mu > 0$ and $\sigma > 0$, replacing   $-\Delta_h \mathbf{B}$ with  $-\Delta \mathbf{B}$ in Equation \eqref{m1} yields the compressible viscous resistive MHD equations. These equations have attracted considerable attention in recent research due to their wide-ranging physical applications, inherent complexity, diverse phenomena, and associated mathematical challenges. For instance, studies such as those by \cite{ChenWang, hongguangyi, HuWang, lihailiang, wuguochun, Xiao} have explored various aspects of these equations.
 \vskip .1in
  %%%%
  \item
  When
$\mu > 0$ and $\sigma = 0$, Equation \eqref{m1} reduces to the compressible viscous non-resistive MHD equations. The lack of magnetic diffusivity poses substantial challenges in establishing well-posedness, particularly for small initial data, and in analyzing stability near a background magnetic field. Wu and Wu \cite{WuWu} systematically proved the global well-posedness of the compressible viscous non-resistive MHD equations in
$\mathbb{R}^2$. Tan and Wang \cite{TW} later established the global existence of smooth solutions for the 3D compressible barotropic viscous non-resistive MHD system in the horizontally infinite flat layer
$\Omega = \mathbb{R}^2 \times (0,1)$. Their proof employs a two-tier energy method applied to the reformulated system in Lagrangian coordinates. Wu and Zhu \cite{WuZhu} were the first to study such a system on periodic domains using pure energy estimates.
Dong, Wu, and Zhai \cite{zhaixiaoping2021arxiv}, \cite{zhaixiaopingnon} proved the global existence of strong solutions for the
$2\frac{1}{2}$-D compressible non-resistive MHD equations with small initial data in Besov and Sobolev spaces, respectively. More recently, Wu and Zhai \cite{zhaixiaoping2022} established the global well-posedness of the compressible viscous non-resistive MHD equations in $\mathbb{T}^3$ under the assumption that the initial magnetic field is sufficiently close to an equilibrium state.
For weak solutions, Li and Sun \cite{LS2D} demonstrated the existence of global weak solutions for the 2D non-resistive compressible MHD equations. This result was later extended by \cite{LiZh1} to accommodate a density-dependent viscosity coefficient and a non-monotone pressure law.

 %%%%%
 \vskip .1in

 \item
When $\mathbf{B} = \mathbf{0}$ and $\mu=\lambda=0$,  \eqref{m1} simplifies to the isentropic compressible Euler equations, which are known to develop finite-time singularities, such as shocks and cusps, from smooth initial data in all dimensions (1D, 2D, and 3D). This phenomenon has been extensively studied in the literature (see \cite{Buc1, Buc2, Buc3, Chr1, Chr2, Luk, Mer, Sid, Xin, Yin}).
\vskip .1in
 %%%%%
  \item
When $\mathbf{B} = \mathbf{0}$ and $\mu>0$, Equation \eqref{m1} simplifies to the isentropic compressible Navier-Stokes equations, a system that has been widely studied in the mathematical literature. Notable contributions include works by Feng \cite{F04}, Hoff \cite{hoff1}, and Xin \cite{Xin4, Xin, xin7}, among others.
%%%%%%%%%

\vskip .1in
 \item
When  $\rho\equiv 0$ and $\mu>0$, $\sigma=0$, equation \eqref{m1} reduces to the viscous non-resistive incompressible MHD system. This system has been extensively studied in the literature, as evidenced by numerous investigations (see, for instance, \cite{abidi, chemin, fefferman1, fefferman2, lijinlu, LXZ, LiZh1, PZZu1, RWZ, Wa1, wujihong1, XZ, ZT}).
           \end{itemize}
\vskip .1in
To the best of the author's knowledge, establishing global solutions for the compressible viscous MHD equations, especially in the absence of magnetic diffusion, remains an open problem, even for small initial data in $\R^3$. This paper aims to address this open problem by investigating the Cauchy problem for the 3D compressible viscous MHD equations, with a specific focus on horizontal magnetic diffusion.
Our main findings are summarized in the following theorem.
\begin{theorem}\label{dingli1}
For any $(\rho_0-\bar\rho,\u_0, \b_0) \in H^3(\R^3)$, there exists a constant $\varepsilon>0$ such that if
\begin{align}\label{}
\|\rho_0-\bar\rho\|_{ H^3}+\|\u_0\|_{ H^3}+\|\b_0\|_{ H^3}\le \varepsilon,
\end{align}
then system \eqref{m1} with  \eqref{long} has a unique global strong solution $(\rho-\bar\rho,\u, \b)$ such that
\begin{align*}
&\rho-\bar{\rho}\in C([0,\infty );H^{3}),\quad \nabla\rho\in L^{2} (\R^+;H^{2}),\nn\\
&\u \in C([0,\infty );H^{3}),\quad\nabla\u\in L^{2} (\R^+;H^{3}),\nn\\
&\b\in C([0,\infty );H^{3}),\quad\nabla_h\b\in L^{2} (\R^+;H^{3}).
\end{align*}
Moreover,  for a constant $C>0$ and any $t > 0$,
\begin{align}\label{}
&\norm{\rho(t)-\bar\rho}{H^3}^2+\norm{\u(t)}{H^3}^2+\norm{\b(t)}{H^3}^2\nn\\
&\quad+\int_0^t\big(\norm{\nabla a}{H^{{2}}}^2+\mu\norm{\nabla\u}{H^3}^2+(\lambda+\mu)\norm{\div\u}{H^3}^2
+\sigma\norm{\nabla_h\b}{H^3}^2\big)\,d\tau\le C\varepsilon.
\end{align}
\end{theorem}

This result is new and its proof is not trivial. There are several difficulties. The first is the whole space $\mathbb R^3$ setting. In contrast to the case of periodic domain, no Poincar\'e type inequality applies in $\mathbb R^3$. The second is the fact that the density equation has no dissipation or damping.  It is well known that solutions to transport equations without such regularization tend to grow in Sobolev spaces.  This paper exploits the coupling and interaction between the density $a(t):=\rho-\bar{\rho}$ and the velocity field to establish global uniform upper bound on $a$ in Sobolev spaces. As demonstrated in Proposition \ref{adis}, this coupling induces a wave structure that provides a smoothing and stabilizing effect on $a$. The third difficulty stems from the lack of vertical dissipation in the magnetic field. This makes the estimate of the nonlinearity more challenging. Various anisotropic inequalities are applied to the triple products from the nonlinearity to obtain optimal derivative distribution.\\

The next section details the proof of Theorem \ref{dingli1}

\section{Proof of Theorem \ref{dingli1}}
In this section, we present the proof of Theorem \ref{dingli1} in three subsections. First, we recall the following anisotropic inequalities, which will play a crucial role in the proof of our main result.
\begin{lemma}\label{ani}
For some constants $C>0$, $i,j,k=1,2,3$ and $i\neq j\neq k$, we have
\begin{align}
&\int_{\R^3}|fgh|\,dx\le C\| f\|_{L^2}^{\f 12}\|\p_1 f\|_{L^2}^{\f 12}\| g\|_{L^2}^{\f 12}\|\p_2 g\|_{L^2}^{\f 12}
\|h\|_{L^2}^{\f 12}\|\p_3 h\|_{L^2}^{\f 12};\label{Ine1}\\
&\int_{\R^3}|fgh|\,dx\le C\|f\|_{L^2}^{\f 14}\|\p_if\|_{L^2}^{\f 14} \|\p_jf\|_{L^2}^{\f 14} \|\p_i\p_jf\|_{L^2}^{\f 14}\|g\|_{L^2}^{\f 12} \|\p_kg\|_{L^2}^{\f12}\|h\|_{L^2}\label{Ine2}.
\end{align}
\end{lemma}

This lemma is essentially established in Lemma 1.2 of \cite{zhuyiadv}; therefore, we omit the details. Next, we derive a priori estimates, which are crucial for establishing global solutions. We start with the fundamental energy estimates.
\subsection{Basic energy estimates} \label{bas}
Denote by $g(\rho)$ the potential energy density, namely
\begin{align}\label{high1}
g(\rho)=\rho\int_{\bar{\rho}}^\rho\frac{P(\tau)-P(\bar{\rho})}{\tau^2}\,d\tau.
\end{align}
For any fixed positive constant $c_0$, if $c_0\le\rho\le c_0^{-1}$, then
\begin{align}\label{high2}
g(\rho)\sim (\rho-\bar{\rho})^2.
\end{align}
The standard basic energy estimate gives
\begin{align}\label{high3}
\frac12\frac{d}{dt}\int_{\R^3}\left(2g(\rho)+\rho|\u|^2+|\b|^2\right)\,dx+\mu\norm{\nabla\u}{L^2}^2+ (\lambda+\mu)\norm{\div\u}{L^2}^2
+\sigma\norm{\nabla_h\b}{L^2}^2=0,
\end{align}
where we have used the following cancellations
\begin{align*}%\label{high4}
&\int_{\R^3}\b\cdot\nabla\b\cdot\u\,dx+\int_{\R^3}\b\cdot\nabla\u\cdot\b\,dx=0,\nn\\
&\int_{\R^3}\b\nabla\b\cdot\u\,dx+\int_{\R^3}\u\cdot\nabla\b\cdot\b\,dx=0.
\end{align*}
Without loss of generality, we assume from now on that $\bar{\rho}=1$, so, we can define
 $$a\stackrel{\mathrm{def}}{=}\rho-\bar{\rho}{=}\rho-1,$$
then \eqref{high3} implies the result in the following proposition.
\begin{proposition}\label{high5}
Let $(a,\u,\b) \in C([0, T];H^3(\R^3))$ be a solution to the  system \eqref{m1}, there holds
\begin{align}\label{high6}
\frac12\frac{d}{dt}\norm{(a,\u,\b)}{L^2}^2+\mu\norm{\nabla\u}{L^2}^2+(\lambda+\mu)\norm{\div\u}{L^2}^2
+\sigma\norm{\nabla_h\b}{L^2}^2=0.
\end{align}
\end{proposition}
\vskip .1in

Next, we are concerned with the higher order nonlinear energy estimates.
\subsection{Higher order energy estimates} \label{high}
In this subsection, we shall prove  the following proposition.

\begin{proposition}\label{ed2}
Let $(a,\u,\b) \in C([0, T];H^3(\R^3))$ be a solution to the  system \eqref{m1}, there holds
\begin{align}\label{highenergy}
&\frac12\frac{d}{dt}\norm{(a,\u,\b)}{H^3}^2+\mu\norm{\nabla\u}{H^3}^2+(\lambda+\mu)\norm{\div\u}{H^3}^2
+\sigma\norm{\nabla_h\b}{H^3}^2 \nn\\
&\quad\le C(\| \u\|_{H^3}+\|(a,\u)\|_{H^3}^2+\| \b\|_{H^3}^4)\| \nabla a\|_{H^2}^2
+C(\|  a\|_{H^3}^2\|  \b\|_{H^3}^2+\|(a,\u,\b)\|_{H^3}^2)\| \nabla \u\|_{H^3}^2.
\end{align}
\end{proposition}
\begin{proof}
 To simplify  the notations, we introduce the new unknowns:
$$ I(a)\stackrel{\mathrm{def}}{=}\frac{a}{1+a},\quad\hbox{and}\quad J(a)\stackrel{\mathrm{def}}{=}\frac{P'(1+a)}{1+a}-1.$$
Then \eqref{m1} with \eqref{long} can be reformulated as
\begin{eqnarray}\label{m3}
\left\{\begin{aligned}
&\partial_t a+ \div\u  =f_1,\\
&\partial_t \u-\mu\Delta\u-(\lambda+\mu)\nabla\div\u+\nabla a
=f_2,\\
&\partial_t \b-\sigma\Delta_h\b=f_3,\\
&\div \b =0,\\
&(a,\u,\b)|_{t=0}=(a_0,\u_0,\b_0),
\end{aligned}\right.
\end{eqnarray}
where $f_1$, $f_2$ and $f_3$  are nonlinear terms defined as
\begin{align*}
f_1\stackrel{\mathrm{def}}{=}&-\u\cdot\nabla a-a\div \u,\nn\\
f_2\stackrel{\mathrm{def}}{=}&-\u\cdot\nabla \u+\b\cdot\nabla\b-\b\nabla\b+J(a)\nabla a
\nn\\
&-I(a)(\mu\Delta\u+(\lambda+\mu)\nabla\div\u)-I(a)(\b\cdot\nabla\b-\b\nabla\b),\nn\\
f_3\stackrel{\mathrm{def}}{=}&-\u\cdot\nabla\b+\b\cdot\nabla\u-\b\div\u.
\end{align*}
 Due to Proposition \ref{high5} and the equivalence of the norm
 \begin{align}\label{high7}
\|(a(t),\u(t),\b(t))\|_{H^{3}}\thickapprox\|(a(t),\u(t),\b(t))\|_{L^2}+\|(a(t),\u(t),\b(t))\|_{\dot{H}^{3}},
\end{align}
it suffices to bound $\|(a(t),\u(t),\b(t))\|_{\dot{H}^{3}}$ in the following argument.
Applying  $\la^3$  to the first three equations of \eqref{m3} and then taking $L^2$ inner product with $(\la^3a, \la^3\u,\la^3\b)$ yields
\begin{align}\label{high8}
&\frac12\frac{d}{dt}\norm{(\la^3a,\la^3\u,\la^3\b)}{L^2}^2 +\mu\norm{\la^3\nabla\u}{L^2}^2+(\lambda+\mu)\norm{\la^3\div\u}{L^2}^2+\sigma\norm{\la^3\nabla_h\b}{L^2}^2 \nn\\
&\quad= \int_{\R^3}\la^{3} f_1\cdot\la^{3} a\,dx+\int_{\R^3}\la^{3} f_2\cdot\la^{3} \u\,dx+\int_{\R^3}\la^{3} f_3\cdot\la^{3} \b\,dx,
\end{align}
due to the following cancellation
\begin{align}\label{high9}
&\int_{\R^3}\la^3 \div\u\cdot\la^3 a\,dx+\int_{\R^3}\la^3 \nabla a\cdot\la^3 \u\,dx=0.
\end{align}
The following task is to estimate the nonlinear terms on the right hand side of \eqref{high8}.
To begin with, we write
\begin{align}\label{high10}
\int_{\R^3}\la^{3} f_1\cdot\la^{3} a\,dx
=&-\int_{\R^3} {\nabla}^3 (\u\cdot \nabla a)\cdot {\nabla}^3 a \; dx -\int_{\R^3} {\nabla}^3 ( a\div\u)\cdot {\nabla}^3 a \; dx\nn\\
=& -\int_{\R^3} \u\cdot \nabla{\nabla}^3  a\cdot {\nabla}^3 a \; dx-\sum_{{k} = 0}^2\int_{\R^3} {\nabla}^{3-{k} } \u\cdot  {\nabla}^{{k} }\nabla a\cdot {\nabla}^3 a \; dx\nn\\
&-\int_{\R^3} {\nabla}^3 ( a\div\u)\cdot {\nabla}^3 a \; dx.
\end{align}
By using the integration by parts and the H\"older inequality, the first term  on the right hand side of \eqref{high10} can be controlled by
\begin{align}\label{high11}
 -\int_{\R^3} \u\cdot \nabla{\nabla}^3  a\cdot {\nabla}^3 a \; dx
 =&\frac12\int_{\R^3} \div\u ({\nabla}^3  a)^2 \; dx\nn\\
 \le& C\|\div\u\|_{L^\infty}\|{\nabla}^3  a\|_{L^2}^2\nn\\
  \le& C\|\u\|_{H^3}\|\nabla a\|_{H^2}^2.
\end{align}
Similarly,
\begin{align}\label{high12}
&-\sum_{{k} = 0}^2\int_{\R^3} {\nabla}^{3-{k} } \u\cdot  {\nabla}^{{k} }\nabla a\cdot {\nabla}^3 a \; dx\nn\\
&\quad\le C\|{\nabla}^{3} \u\|_{L^2}\|\nabla a\|_{L^\infty}\|{\nabla}^3 a\|_{L^2}\nn\\
&\qquad+ C\|{\nabla}^{2} \u\|_{L^3}\|\nabla^2 a\|_{L^6}\|{\nabla}^3 a\|_{L^2}+C\|{\nabla} \u\|_{L^\infty}\|{\nabla}^2\nabla a\|_{L^2}\|{\nabla}^3 a\|_{L^2}\nn\\
&\quad\le C\|\u\|_{H^3}\|\nabla a\|_{H^2}^2.
\end{align}
For the last term in \eqref{high10}, it's not hard to check that
\begin{align}\label{}
 -\int_{\R^3} {\nabla}^3 ( a\div\u)\cdot {\nabla}^3 a \; dx
=&-\int_{\R^3} a{\nabla}^3 \div\u {\nabla}^3 a \; dx -\int_{\R^3} {\nabla}^{3 } a \div\u\cdot {\nabla}^3 a \; dx\nn\\
&-\sum_{{k} = 1}^2\mathbb{C}_{3}^k\int_{\R^3} {\nabla}^{3-{k} } a  {\nabla}^{{k} }\div\u\cdot {\nabla}^3 a \; dx.
\end{align}
Now, with an aid of the H\"older inequality and the Young inequality, there hold
\begin{align*}%\label{}
-\int_{\R^3} a{\nabla}^3 \div\u {\nabla}^3 a \; dx
\le& C\| a\|_{L^\infty}\|{\nabla}^{3 }\div\u\|_{L^2}\|{\nabla}^3 a\|_{L^2}\nn\\
\le& \frac{\mu}{16}\|{\nabla}^{3 }\div\u\|_{L^2}^2+C\| a\|_{L^\infty}^2\|{\nabla}^3 a\|_{L^2}^2\nn\\
\le&\frac{\mu}{16}\|\nabla\u\|_{H^3}^2+C\| a\|_{H^2}^2\|\nabla a\|_{H^2}^2;\nn\\
-\int_{\R^3} {\nabla}^{3 } a \div\u\cdot {\nabla}^3 a \; dx
\le& C\| \div\u\|_{L^\infty}\|{\nabla}^{3 }a\|_{L^2}^2
\le C\| \u\|_{H^3}\|\nabla a\|_{H^2}^2;\nn\\
\end{align*}
and
\begin{align*}%\label{}
&-\sum_{{k} = 1}^2\mathbb{C}_{3}^k\int_{\R^3} {\nabla}^{3-{k} } a  {\nabla}^{{k} }\div\u\cdot {\nabla}^3 a \; dx\nn\\
&\quad\le
C\|{\nabla}^{2 } a\|_{L^6}\|{\nabla}\div\u\|_{L^3}\|{\nabla}^3 a\|_{L^2}
+C\|{\nabla} a\|_{L^\infty}\|{\nabla}^{2 }\div\u\|_{L^2}\|{\nabla}^3 a\|_{L^2}\nn\\
&\quad\le
C\|{\nabla}^{2 } a\|_{H^1}\|{\nabla}\div\u\|_{H^1}\|{\nabla}^3 a\|_{L^2}
+C\|{\nabla} a\|_{H^2}\|{\nabla}^{2 }\div\u\|_{L^2}\|{\nabla}^3 a\|_{L^2}\nn\\
&\quad\le C\| a\|_{H^3}\|\nabla \u\|_{H^2}\|\nabla a\|_{H^2}\nn\\
&\quad\le\frac{\mu}{16}\|\nabla\u\|_{H^3}^2+C\| a\|_{H^3}^2\|\nabla a\|_{H^2}^2
\end{align*}
from which, we can get
\begin{align}\label{high13}
 -\int_{\R^3} {\nabla}^3 ( a\div\u)\cdot {\nabla}^3 a \; dx
\le& \frac{\mu}{8}\|\nabla\u\|_{H^3}^2+C\| a\|_{H^3}^2\|\nabla a\|_{H^2}^2+C\| \u\|_{H^3}\|\nabla a\|_{H^2}^2.
\end{align}
The combination of \eqref{high11}, \eqref{high12}, and \eqref{high13} leads to
\begin{align}\label{high14}
\int_{\R^3}\la^{3} f_1\cdot\la^{3} a\,dx
\le \frac{\mu}{4}\|\nabla\u\|_{H^3}^2+C(\| a\|_{H^3}^2+\| \u\|_{H^3})\|\nabla a\|_{H^2}^2.
\end{align}
In order to bound   terms in $f_2$, we  write
\begin{align}\label{high15}
\int_{\R^3}\la^{3} f_2\cdot\la^{3} \u\,dx=\sum_{i=1}^7D_i
\end{align}
with
\begin{align}\label{high169}
D_1\stackrel{\mathrm{def}}{=}&-\int_{\R^3}\la^{3} (\u\cdot\nabla \u)\cdot\la^{3} \u\,dx,\quad
\hspace{0.47cm} D_2\stackrel{\mathrm{def}}{=}\int_{\R^3}\la^{3} (\b\cdot\nabla\b)\cdot\la^{3} \u\,dx,\nn\\
D_3\stackrel{\mathrm{def}}{=}&-\int_{\R^3}\la^{3} (\b\nabla\b)\cdot\la^{3} \u\,dx,\quad
\hspace{0.62cm} D_4\stackrel{\mathrm{def}}{=}\int_{\R^3}\la^{3} (J(a)\nabla a)\cdot\la^{3} \u\,dx,\nn\\
D_5\stackrel{\mathrm{def}}{=}&\int_{\R^3}\la^{3} (I(a)(\b\nabla\b))\cdot\la^{3} \u\,dx,\quad
D_6\stackrel{\mathrm{def}}{=}-\int_{\R^3}\la^{3} (I(a)(\b\cdot\nabla\b))\cdot\la^{3} \u\,dx,\nn\\
D_7\stackrel{\mathrm{def}}{=}&-\int_{\R^3}\la^{3} (I(a)(\mu\Delta\u+(\lambda+\mu)\nabla\div\u))\cdot\la^{3} \u\,dx.
\end{align}
Bounding the nonlinear terms involving composition functions in
$D_4$, $D_5$, $D_6$ and $D_7$
  requires a more intricate approach. To facilitate this analysis, we consistently assume that
\begin{equation}\label{eq:smallad }
\sup_{t \in \mathbb{R}_+, x \in \mathbb{R}^3} |a(t, x)| \leq \frac{1}{2}.
\end{equation}
This assumption is justified by the embedding
$H^2(\R^3)\hookrightarrow L^\infty(\R^3)$, which ensures that \eqref{eq:smallad } holds provided the solution constructed here has a sufficiently small norm in
$H^2(\R^3)$. Consequently, applying the composite function lemma (see, for instance, Lemma 2.6 in \cite{zhaixiaoping2022}), we obtain the following estimate:
\begin{equation}\label{eq:smalla}
\|(I(a),J(a))\|_{H^s}\le C\|a\|_{H^s}, \quad\hbox{for any $s>0$}.
\end{equation}
Now we turn to estimate $D_1, \cdot\cdot\cdot, D_7$.
 At first, it follows from the H\"older inequality and the Young inequality that
\begin{align}\label{high16}
D_1 =& -\int_{\R^3} {\nabla}^3 (\u\cdot \nabla \u)\cdot {\nabla}^3 \u \; dx\nn\\
\le& C\|{\nabla}^2 (\u\cdot \nabla \u)\|_{L^2}\|{\nabla}^4 \u\|_{L^2}\nn\\
\le&\frac{\mu}{16}\|{\nabla}^4 \u\|_{L^2}^2+ C\|\u\cdot \nabla \u\|_{H^2}^2\nn\\
\le&\frac{\mu}{16}\|\nabla\u\|_{H^3}^2+ C\|\u\|_{H^3}^2\|\nabla \u\|_{H^3}^2.
\end{align}
 In the same fashion, one can show that
\begin{align}\label{high20}
D_4=&\int_{\R^3} {\nabla}^3 (J(a)\nabla a)\cdot {\nabla}^3 \u \; dx\nn\\
=&-\int_{\R^3} {\nabla}^2 (J(a)\nabla a)\cdot {\nabla}^4 \u \; dx\nn\\
\le&C\| J(a)\nabla a\|_{H^2}\| {\nabla} \u\|_{H^3}\nn\\
\le&\frac{\mu}{16}\| \nabla \u\|_{H^3}^2+C\|  a\|_{H^3}^2\| \nabla a\|_{H^2}^2,
\end{align}
and
\begin{align}\label{high21}
D_7=&\int_{\R^3} {\nabla}^3 (I(a)(\mu\Delta\u+(\lambda+\mu)\nabla\div\u))\cdot {\nabla}^3 \u \; dx\nn\\
=&-\int_{\R^3} {\nabla}^2 (I(a)(\mu\Delta\u+(\lambda+\mu)\nabla\div\u))\cdot {\nabla}^4 \u \; dx\nn\\
\le&C\| I(a)(\mu\Delta\u+(\lambda+\mu)\nabla\div\u)\|_{H^2}\| {\nabla} \u\|_{H^3}\nn\\
\le&\frac{\mu}{16}\| \nabla \u\|_{H^3}^2+C\|  a\|_{H^3}^2\| \nabla \u\|_{H^3}^2.
\end{align}
Next, we focus on the rest terms $D_2,$ $D_3,$ $D_5$, and $D_6.$ At first, by Leibniz's rule, we split $D_2$ into two terms
\begin{align}\label{high17}
D_2& =  \sum_{k = 0}^2 \mathbb{C}_{3}^k\int_{\R^3}{\nabla}^{3-k} \b \cdot \na{\nabla}^{k} \b\cdot {\nabla}^3 \u \, dx
+ \int_{\R^3} \b \cdot \nabla {\nabla}^3\b \cdot {\nabla}^3\u \, dx.
\end{align}
It then follows from Lemma \ref{ani} directly that
\begin{align*}
&\sum_{k = 0}^2 \mathbb{C}_{3}^k\int_{\R^3}{\nabla}^{3-k} \b \cdot \na{\nabla}^{k} \b\cdot {\nabla}^3 \u \, dx\nn\\
&\quad\le C\sum_{k = 0}^2\| {\nabla}^{3-k} \b\|_{L^2}^{\f 12}\|\p_1 {\nabla}^{3-k} \b\|_{L^2}^{\f 12}\| \na{\nabla}^{k} \b\|_{L^2}^{\f 12}\|\p_2 \na{\nabla}^{k} \b\|_{L^2}^{\f 12}
\|{\nabla}^3 \u\|_{L^2}^{\f 12}\|\p_3 {\nabla}^3 \u\|_{L^2}^{\f 12}\nn\\
&\quad\le  C\|  \b\|_{H^3}^{\f 12}\| \nabla_h \b\|_{H^3}^{\f 12}\|  \b\|_{H^3}^{\f 12}\| \nabla_h \b\|_{H^3}^{\f 12}\| \nabla \u\|_{H^2}^{\f 12}\| \nabla \u\|_{H^3}^{\f 12}\nn\\
&\quad\le  C\|  \b\|_{H^3}\| \nabla_h \b\|_{H^3}\| \nabla \u\|_{H^3}\nn\\
&\quad\le  \frac{\sigma}{16}\| \nabla_h \b\|_{H^3}^2+C\|  \b\|_{H^3}^2\| \nabla \u\|_{H^3}^2
\end{align*}
from which  one has
\begin{align}\label{high18}
D_2\le  \frac{\sigma}{16}\| \nabla_h \b\|_{H^3}^2+C\|  \b\|_{H^3}^2\| \nabla \u\|_{H^3}^2
+ \int_{\R^3} \b \cdot \nabla {\nabla}^3\b \cdot {\nabla}^3\u \, dx.
\end{align}
Now we turn to the next term  $D_3$. We use the integration by parts and Lemma \ref{ani} to get
\begin{align}\label{high19}
&D_3
=\int_{\R^3} {\nabla}^3 (\frac12 |\b|^2)\cdot {\nabla}^3 \div\u \; dx\nn\\
&\quad\le C\Big|\int_{\R^3}(\b{\nabla}^{3} \b+{\nabla}\b{\nabla}^{2}\b  )\cdot {\nabla}^3 \div\u \, dx\Big|\nn\\
&\quad\le C\|\b\|_{L^2}^{\f 14}\|\p_1\b\|_{L^2}^{\f 14} \|\p_3\b\|_{L^2}^{\f 14} \|\p_1\p_3\b\|_{L^2}^{\f 14}\|{\nabla}^{3} \b\|_{L^2}^{\f 12} \|\p_2{\nabla}^{3} \b\|_{L^2}^{\f12}\|{\nabla}^3 \div\u\|_{L^2}\nn\\
&\qquad+C\|{\nabla} \b\|_{L^2}^{\f 14}\|\p_1{\nabla} \b\|_{L^2}^{\f 14} \|\p_3{\nabla} \b\|_{L^2}^{\f 14} \|\p_1\p_3{\nabla} \b\|_{L^2}^{\f 14}\|{\nabla}^{2} \b\|_{L^2}^{\f 12} \|\p_2{\nabla}^{2} \b\|_{L^2}^{\f12}\|{\nabla}^3 \div\u\|_{L^2}\nn\\
&\quad\le C\| \nabla \u\|_{H^3}\| \nabla_h \b\|_{H^3}\| \b\|_{H^3}\nn\\
&\quad\le \frac{\sigma}{16}\| \nabla_h \b\|_{H^3}^2+C\|  \b\|_{H^3}^2\| \nabla \u\|_{H^3}^2.
\end{align}
In analogy with before, $D_5$ could be divided into two terms
\begin{align}\label{high22}
&D_5=\int_{\R^3} {\nabla}^3 (I(a)\b\nabla\b)\cdot {\nabla}^3 \u \; dx\nn\\
&\quad=\sum_{k = 0}^2 \mathbb{C}_{3}^k\int_{\R^3}{\nabla}^{3-k} I(a){\nabla}^{k}(\b\nabla\b) \cdot {\nabla}^3 \u \; dx+\int_{\R^3} I(a){\nabla}^3 (\b\nabla\b)\cdot {\nabla}^3 \u \; dx\nn\\
&\quad\stackrel{\mathrm{def}}{=}D_{5,1}+D_{5,2}.
\end{align}
With an aid of the H\"older inequality and the Young inequality, there holds
\begin{align}\label{high23}
D_{5,1}=&\sum_{k = 0}^2 \mathbb{C}_{3}^k\int_{\R^3}{\nabla}^{3-k} I(a){\nabla}^{k}(\b\nabla\b) \cdot {\nabla}^3 \u \; dx\nn\\
\le&\sum_{k = 0}^2\|{\nabla}^{3-k} I(a)\|_{L^2}\|{\nabla}^{k}(\b\nabla\b)\|_{L^\infty}\|{\nabla}^{3} \u\|_{L^2}\nn\\
\le& C\|{\nabla}^3 a\|_{L^2}\|\b\nabla\b\|_{L^\infty}\|{\nabla}^{3} \u\|_{L^2}+C\|{\nabla}^2 a\|_{L^6}\|{\nabla}(\b\nabla\b)\|_{L^2}\|{\nabla}^{3} \u\|_{L^2}\nn\\
&\quad+C\|{\nabla} a\|_{L^\infty}\|{\nabla}^2(\b\nabla\b)\|_{L^2}\|{\nabla}^{3} \u\|_{L^2}\nn\\
\le& C\| \nabla a\|_{H^2}\| \b\|_{H^3}^2\| \nabla \u\|_{H^3}
\nn\\
\le&\frac{\mu}{16}\| \nabla \u\|_{H^3}^2+C\| \b\|_{H^3}^4\| \nabla a\|_{H^2}^2.
\end{align}
By using the integration by parts, we further write $D_{5,2}$ into
\begin{align}\label{high24}
D_{5,2}
=&-\int_{\R^3}\nabla I(a){\nabla}^3 (\frac12 |\b|^2)\cdot {\nabla}^3 \u \; dx-\int_{\R^3} I(a){\nabla}^3 (\frac12 |\b|^2)\cdot {\nabla}^3 \div\u \; dx.
\end{align}
Now, it follows from the H\"older inequality and \eqref{high19} that
\begin{align*}%\label{high25}
-\int_{\R^3}\nabla I(a){\nabla}^3 (\frac12 |\b|^2)\cdot {\nabla}^3 \u \; dx
\le& C\|\nabla I(a)\|_{L^\infty}\Big|\int_{\R^3}(\b{\nabla}^{3} \b+{\nabla}\b{\nabla}^{2}\b  )\cdot {\nabla}^3\u \, dx\Big|\nn\\
\le& C\|\nabla I(a)\|_{H^2}\| \nabla \u\|_{H^3}\| \nabla_h \b\|_{H^3}\| \b\|_{H^3}\nn\\
\le&\frac{\sigma}{16}\| \nabla_h \b\|_{H^3}^2+C\|  a\|_{H^3}^2\|  \b\|_{H^3}^2\| \nabla \u\|_{H^3}^2,
\end{align*}
and
\begin{align*}%\label{high26}
-\int_{\R^3} I(a){\nabla}^3 (\frac12 |\b|^2)\cdot {\nabla}^3 \div\u \; dx
\le& C\| I(a)\|_{L^\infty}\Big|\int_{\R^3} {\nabla}^3 (\frac12 |\b|^2)\cdot {\nabla}^3 \div\u \; dx\Big|\nn\\
\le&\frac{\sigma}{16}\| \nabla_h \b\|_{H^3}^2+C\|  a\|_{H^3}^2\|  \b\|_{H^3}^2\| \nabla \u\|_{H^3}^2.
\end{align*}
Taking the above two estimates  back to \eqref{high24} and combining with  \eqref{high23}, we obtain that
\begin{align*}%\label{high27}
D_5
\le&\frac{\sigma}{8}\| \nabla_h \b\|_{H^3}^2+\frac{\mu}{16}\| \nabla \u\|_{H^3}^2+C\| \b\|_{H^3}^4\| \nabla a\|_{H^2}^2
+C\|  a\|_{H^3}^2\|  \b\|_{H^3}^2\| \nabla \u\|_{H^3}^2.
\end{align*}
The term $D_6$ can be estimated similarly, thus, we can get
\begin{align}\label{high28}
\int_{\R^3}\la^{3} f_2\cdot\la^{3} \u\,dx
\le& \int_{\R^3} \b \cdot \nabla {\nabla}^3\b \cdot {\nabla}^3\u \, dx+\frac{\mu}{4}\| \nabla \u\|_{H^3}^2+\frac{\sigma}{4}\| \nabla_h \b\|_{H^3}^2\\
&+C(\|  a\|_{H^3}^2+\| \b\|_{H^3}^4)\| \nabla a\|_{H^2}^2
+C(\|  a\|_{H^3}^2\|  \b\|_{H^3}^2+\| ( a,\u,\b)\|_{H^3}^2)\| \nabla \u\|_{H^3}^2.\nn
\end{align}

Finally, we deal with the term $\int_{\R^3}\la^{3} f_3\cdot\la^{3} \b\,dx$ in \eqref{high8}. Recall that
$$f_3\stackrel{\mathrm{def}}{=}-\u\cdot\nabla\b+\b\cdot\nabla\u-\b\div\u,$$
we can then use  Leibniz's rule again to write
\begin{align}\label{high29}
 -\int_{\R^3} {\nabla}^3 (\u\cdot \nabla \b)\cdot {\nabla}^3 \b \; dx
=& -\int_{\R^3} \u\cdot \nabla{\nabla}^3  \b\cdot {\nabla}^3 \b \; dx-\sum_{{k} = 0}^2\int_{\R^3} {\nabla}^{3-{k} } \u\cdot  {\nabla}^{{k} }\nabla \b\cdot {\nabla}^3 \b \; dx.
\end{align}
Thanks to integration by parts and Lemma \ref{ani}, there holds
\begin{align}\label{high30}
 -\int_{\R^3} \u\cdot \nabla{\nabla}^3  \b\cdot {\nabla}^3 \b \; dx
 =& \frac12\int_{\R^3} \div\u ({\nabla}^3  \b)^2 \; dx\nn\\
 \le&  C\| {\nabla}^3  \b\|_{L^2}^{\f 12}\|\p_1 {\nabla}^3  \b\|_{L^2}^{\f 12}\| {\nabla}^3  \b\|_{L^2}^{\f 12}\|\p_2 {\nabla}^3  \b\|_{L^2}^{\f 12}
\|\div\u\|_{L^2}^{\f 12}\|\p_3 \div\u\|_{L^2}^{\f 12}\nn\\
\le&  C\|\nabla\u\|_{H^3}\|\b\|_{H^3}\|\nabla_h\b\|_{H^3}\nn\\
\le&\frac{\sigma}{16}\| \nabla_h \b\|_{H^3}^2+C\|  \b\|_{H^3}^2\| \nabla \u\|_{H^3}^2.
\end{align}
For the last term in \eqref{high29}, it follows from  Lemma \ref{ani} once again and the Young inequality that
\begin{align}\label{high31}
&-\sum_{{k} = 0}^2\int_{\R^3} {\nabla}^{3-{k} } \u\cdot  {\nabla}^{{k} }\nabla \b\cdot {\nabla}^3 \b \; dx\nn\\
&\quad\le  C\sum_{{k} = 0}^2\| {\nabla}^{{k} }\nabla \b\|_{L^2}^{\f 12}\|\p_1 {\nabla}^{{k} }\nabla \b\|_{L^2}^{\f 12}\| {\nabla}^3  \b\|_{L^2}^{\f 12}\|\p_2 {\nabla}^3  \b\|_{L^2}^{\f 12}
\|{\nabla}^{3-{k} } \u\|_{L^2}^{\f 12}\|\p_3 {\nabla}^{3-{k} } \u\|_{L^2}^{\f 12}\nn\\
&\quad\le  C\|\nabla\u\|_{H^3}\|\b\|_{H^3}\|\nabla_h\b\|_{H^3}\nn\\
&\quad\le\frac{\sigma}{16}\| \nabla_h \b\|_{H^3}^2+C\|  \b\|_{H^3}^2\| \nabla \u\|_{H^3}^2.
\end{align}
Hence, the combination \eqref{high30} with \eqref{high31} gives rise to
\begin{align}\label{high32}
 -\int_{\R^3} {\nabla}^3 (\u\cdot \nabla \b)\cdot {\nabla}^3 \b \; dx
\le\frac{\sigma}{8}\| \nabla_h \b\|_{H^3}^2+C\|  \b\|_{H^3}^2\| \nabla \u\|_{H^3}^2.
\end{align}
Similarly, for the second term in $f_3$, there holds the following equality
\begin{align}\label{high33}
\int_{\R^3} {\nabla}^3 (\b\cdot \nabla \u)\cdot {\nabla}^3 \b \; dx& =  \sum_{k = 0}^2 \mathbb{C}_{3}^k\int_{\R^3}{\nabla}^{3-k} \b \cdot \na{\nabla}^{k} \u\cdot {\nabla}^3 \b \, dx
+ \int_{\R^3} \b \cdot \nabla {\nabla}^3\u \cdot {\nabla}^3\b \, dx.
\end{align}
In view of Lemma \ref{ani}, for the first term on the right hand side of \eqref{high33}, one has
\begin{align*}%\label{high34}
&\sum_{k = 0}^2 \mathbb{C}_{3}^k\int_{\R^3}{\nabla}^{3-k} \b \cdot \na{\nabla}^{k} \u\cdot {\nabla}^3 \b \, dx\nn\\
&\quad\le C\sum_{k = 0}^2\| {\nabla}^{3-k} \b\|_{L^2}^{\f 12}\|\p_1 {\nabla}^{3-k} \b\|_{L^2}^{\f 12}
\|{\nabla}^3 \b\|_{L^2}^{\f 12}\|\p_2 {\nabla}^3 \b\|_{L^2}^{\f 12}\| \na{\nabla}^{k} \u\|_{L^2}^{\f 12}\|\p_3 \na{\nabla}^{k} \u\|_{L^2}^{\f 12}\nn\\
&\quad\le C\|  \b\|_{H^3}\| \nabla_h \b\|_{H^3}\| \nabla \u\|_{H^3}\nn\\
&\quad\le \frac{\sigma}{16}\| \nabla_h \b\|_{H^3}^2+C\|  \b\|_{H^3}^2\| \nabla \u\|_{H^3}^2
\end{align*}
 which implies
\begin{align}\label{high35}
\int_{\R^3} {\nabla}^3 (\b\cdot \nabla \u)\cdot {\nabla}^3 \b \; dx\le \frac{\sigma}{16}\| \nabla_h \b\|_{H^3}^2+C\|  \b\|_{H^3}^2\| \nabla \u\|_{H^3}^2
+ \int_{\R^3} \b \cdot \nabla {\nabla}^3\u \cdot {\nabla}^3\b \, dx.
\end{align}
For the last term in $f_3$, we use  Leibniz's rule again to write
\begin{align}\label{high36}
\int_{\R^3} {\nabla}^3 (\b\div\u)\cdot {\nabla}^3 \b \; dx& =  \sum_{k = 0}^2 \mathbb{C}_{3}^k\int_{\R^3}{\nabla}^{3-k} \b {\nabla}^{k} \div\u\cdot {\nabla}^3 \b \, dx
+ \int_{\R^3} \b  {\nabla}^3\div \u \cdot {\nabla}^3\b \, dx.
\end{align}
Then, from Lemma \ref{ani}, the first term on the right hand side of \eqref{high36} can be bounded by
\begin{align}\label{high37}
&\sum_{k = 0}^2 \mathbb{C}_{3}^k\int_{\R^3}{\nabla}^{3-k} \b {\nabla}^{k} \div\u\cdot {\nabla}^3 \b \, dx\nn\\
&\quad\le C\sum_{k = 0}^2\| {\nabla}^{3-k} \b\|_{L^2}^{\f 12}\|\p_1 {\nabla}^{3-k} \b\|_{L^2}^{\f 12}
\|{\nabla}^3 \b\|_{L^2}^{\f 12}\|\p_2 {\nabla}^3 \b\|_{L^2}^{\f 12}\| {\nabla}^{k} \div\u\|_{L^2}^{\f 12}\|\p_3 {\nabla}^{k} \div\u\|_{L^2}^{\f 12}\nn\\
&\quad\le C\|  \b\|_{H^3}\| \nabla_h \b\|_{H^3}\| \nabla \u\|_{H^3}\nn\\
&\quad\le\frac{\sigma}{16}\| \nabla_h \b\|_{H^3}^2+C\|  \b\|_{H^3}^2\| \nabla \u\|_{H^3}^2.
\end{align}
By using the second inequality in Lemma \ref{ani}, the last term in \eqref{high36} can be bounded by
\begin{align}\label{high38}
\int_{\R^3} \b  {\nabla}^3\div\u \cdot {\nabla}^3\b \, dx
\le& C\|\b\|_{L^2}^{\f 14}\|\p_1\b\|_{L^2}^{\f 14} \|\p_3\b\|_{L^2}^{\f 14} \|\p_1\p_3\b\|_{L^2}^{\f 14}\|{\nabla}^3\b\|_{L^2}^{\f 12} \|\p_2{\nabla}^3\b\|_{L^2}^{\f12}\|{\nabla}^3\div\u\|_{L^2}\nn\\
\le& C\|  \b\|_{H^3}\| \nabla_h \b\|_{H^3}\| \nabla \u\|_{H^3}\nn\\
\le& \frac{\sigma}{16}\| \nabla_h \b\|_{H^3}^2+C\|  \b\|_{H^3}^2\| \nabla \u\|_{H^3}^2
\end{align}
from which and \eqref{high37}, we can get
\begin{align}\label{high39}
\int_{\R^3} {\nabla}^3 (\b\div\u)\cdot {\nabla}^3 \b \; dx\le& \frac{\sigma}{8}\| \nabla_h \b\|_{H^3}^2+C\|  \b\|_{H^3}^2\| \nabla \u\|_{H^3}^2.
\end{align}
Combining with \eqref{high32}, \eqref{high35}, and \eqref{high39}, we have
\begin{align}\label{high40}
\int_{\R^3}\la^{3} f_3\cdot\la^{3} \b\,dx
\le&\frac{5\sigma}{16}\| \nabla_h \b\|_{H^3}^2+C\|  \b\|_{H^3}^2\| \nabla \u\|_{H^3}^2
+ \int_{\R^3} \b \cdot \nabla {\nabla}^3\u \cdot {\nabla}^3\b \, dx.
\end{align}

Inserting \eqref{high14}, \eqref{high28}, and \eqref{high40} into  \eqref{high8} and using the following cancellation
\begin{align*}%\label{high41}
 \int_{\R^3} \b \cdot \nabla {\nabla}^3\b \cdot {\nabla}^3\u \, dx+ \int_{\R^3} \b \cdot \nabla {\nabla}^3\u \cdot {\nabla}^3\b \, dx=0,
\end{align*}
we can arrive at \eqref{highenergy}. Consequently, we complete the proof  the proposition.
\end{proof}

\vskip .1in
By the line, we now turn to the dissipation of the density. More precisely, we have the following proposition.
\begin{proposition} \label{adis}
Let $(a,\u,\b) \in C([0, T];H^3(\R^3))$ be a solution to the  system \eqref{m3}, there holds
\begin{align}\label{yiming1}
	&\norm{\nabla a}{H^{{2}}}^2+\sum_{k=0}^2\frac{d}{dt}{\int_{\R^3}{{\nabla}^{k}} \u\cdot{{\nabla}^{k}}\nabla a\,dx}\nn\\
&\quad\le C\norm{\nabla \u}{H^3}^2+C\norm{(a,\b)}{H^3}^2\norm{\nabla a}{H^2}^2
+C\norm{(a,\u)}{H^3}^2\norm{\nabla \u}{H^3}^2 +C\norm{(a,\b)}{H^3}^2\norm{\nabla_h \b}{H^3}^2.
\end{align}
\end{proposition}

\begin{proof}
We first recall from the second equation of \eqref{m3} that
$$\nabla a
=-\partial_t \u+\mu\Delta\u+(\lambda+\mu)\nabla\div\u+f_2.
$$
Applying operator ${\nabla}^{k}(k=0,1,2)$ to the above equation, taking inner product with
${{\nabla}^{k}}\nabla a$, we then obtain
\begin{align}\label{yiming2}
\norm{{{\nabla}^{k}}\nabla a}{L^2}^2
=&-{\int_{\R^3}{{\nabla}^{k}}\partial_t \u\cdot{{\nabla}^{k}}\nabla a\,dx}\nn\\
&+{\int_{\R^3}{{\nabla}^{k}} (\mu\Delta\u+(\lambda+\mu)\nabla\div\u)\cdot{{\nabla}^{k}}\nabla a\,dx}+{\int_{\R^3}{{\nabla}^{k}} f_2\cdot{{\nabla}^{k}}\nabla a\,dx}\nn\\
\stackrel{\mathrm{def}}{=}&M_{1}+M_{2}+M_{3}.
\end{align}
Now  we estimate  nonlinear terms on the right hand side of \eqref{yiming2}. A simple calculation gives another convenient form of $M_1$. Indeed
\begin{align}\label{yiming3}
M_{1}=&-{\int_{\R^3}{{\nabla}^{k}}\partial_t \u\cdot{{\nabla}^{k}}\nabla
a\,dx}\nn\\
=&-\frac{d}{dt}{\int_{\R^3}{{\nabla}^{k}} \u\cdot{{\nabla}^{k}}\nabla a\,dx}-{\int_{\R^3}{{\nabla}^{k}}\div\u\cdot{{\nabla}^{k}} \partial_t a\,dx}\nn\\
=&-\frac{d}{dt}{\int_{\R^3}{{\nabla}^{k}} \u\cdot{{\nabla}^{k}}\nabla a\,dx}
+{\int_{\R^3}{{\nabla}^{k}}\div\u\cdot{{\nabla}^{k}}\div\u \,dx}-{\int_{\R^3}{{\nabla}^{k}}\div\u\cdot{{\nabla}^{k}} f_1\,dx}\nn\\
=& -\frac{d}{dt}{\int_{\R^3}{{\nabla}^{k}} \u\cdot{{\nabla}^{k}}\nabla a\,dx}+\norm{{{\nabla}^{k}}\div\u}{L^2}^2-{\int_{\R^3}{{\nabla}^{k}}\div\u\cdot{{\nabla}^{k}} f_1\,dx}.
\end{align}
For $k=0,1,2$, with an aid of the H\"older inequality and the Young inequality, there hold
\begin{align}\label{yiming4}
	-{\int_{\R^3}{{\nabla}^{k}}\div\u\cdot{{\nabla}^{k}} (\u\cdot\nabla a)\,dx}
	\le& C\norm{\div\u}{H^{{2}}} \norm{\u}{H^{{2}}}\norm{\nabla a}{H^{{2}}}\nn\\
	\le& \,\frac1{16}\norm{\nabla a}{H^{{2}}}^2+C\norm{\u}{H^{{2}}}^2\norm{\div\u}{H^{{2}}}^2\nn\\
\le& \,\frac1{16}\norm{\nabla a}{H^{{2}}}^2+C\norm{\u}{H^{{3}}}^2\norm{\nabla\u}{H^{{3}}}^2,
\end{align}
and
\begin{align}\label{yiming5}
	-{\int_{\R^3}{{\nabla}^{k}}\div\u\cdot{{\nabla}^{k}} (a\div\u)\,dx}
	\le& C\norm{\div\u}{H^{{2}}} \norm{a}{H^{{2}}}\norm{\div\u}{H^{{2}}}\nn\\
	\le& C\norm{a}{H^{{3}}}^2\norm{\nabla\u}{H^{{3}}}^2.
\end{align}

 Inserting the above two estimates into \eqref{yiming3}, we can get
\begin{align}\label{yiming6}
	M_{1}
	\le&-\frac{d}{dt}{\int_{\R^3}{{\nabla}^{k}} \u\cdot{{\nabla}^{k}}\nabla a\,dx}+\frac1{16}\norm{\nabla a}{H^{{2}}}^2+C\norm{\nabla\u}{H^{{3}}}^2+C\norm{(a,\u)}{H^{3}}^2\norm{\nabla\u}{H^{{3}}}^2.
\end{align}
For the second term $M_{2}$ , it follows from the H\"older inequality and the Young inequality directly that
\begin{align}\label{yiming7}
M_{2}\le&\frac1{16}\norm{{\nabla}^{k}\nabla
a}{L^2}^2+C\norm{\nabla\u}{H^{k+1}}^2\nn\\
\le&\frac1{16}\norm{\nabla a}{H^{{2}}}^2+C\norm{\nabla\u}{H^{3}}^2.
\end{align}
At last, we deal with the term $M_3.$
We first recall that
\begin{align}\label{yiming9}
f_2\stackrel{\mathrm{def}}{=}&-\u\cdot\nabla \u+\b\cdot\nabla\b-\b\nabla\b+J(a)\nabla a
\nn\\
&-I(a)(\mu\Delta\u+(\lambda+\mu)\nabla\div\u)-I(a)(\b\cdot\nabla\b-\b\nabla\b).
\end{align}
Then, according to the H\"older inequality and the Young inequality, we have
\begin{align}\label{yiming15}
		{\int_{\R^3}{{\nabla}^{k}} (\u\cdot\nabla \u)\cdot{{\nabla}^{k}}\nabla a\,dx}
	\le&\frac1{16}\norm{{{\nabla}^{k}}\nabla
	a}{L^2}^2+\norm{{{\nabla}^{k}}(\u\cdot\nabla \u)}{L^2}^2\nn\\
\le&\frac1{16}\norm{\nabla a}{H^{{2}}}^2+C\norm{\u}{H^3}^2\norm{\nabla \u}{H^3}^2,
\end{align}
\begin{align}\label{yiming16}
		{\int_{\R^3}{{\nabla}^{k}} (J(a)\nabla a)\cdot{{\nabla}^{k}}\nabla a\,dx}
	\le&\frac1{16}\norm{{{\nabla}^{k}}\nabla
	a}{L^2}^2+\norm{{{\nabla}^{k}}(J(a)\nabla a)}{L^2}^2\nn\\
\le&\frac1{16}\norm{\nabla a}{H^{{2}}}^2+C\norm{a}{H^3}^2\norm{\nabla a}{H^2}^2,
\end{align}
and
\begin{align}\label{yiming17}
		&{\int_{\R^3}{{\nabla}^{k}} (I(a)(\mu\Delta\u+(\lambda+\mu)\nabla\div\u))\cdot{{\nabla}^{k}}\nabla a\,dx}\nn\\
	&\quad\le \frac1{16}\norm{{{\nabla}^{k}}\nabla
	a}{L^2}^2+\norm{{{\nabla}^{k}}(I(a)(\Delta\u+\nabla\div\u)}{L^2}^2\nn\\
&\quad\le \frac1{16}\norm{{{\nabla}^{k}}\nabla
	a}{L^2}^2+C\norm{I(a)}{H^3}^2\norm{\Delta\u+\nabla\div\u}{H^2}^2\nn\\
&\quad\le \frac1{16}\norm{\nabla a}{H^{{2}}}^2+ C\norm{a}{H^3}^2\norm{\nabla \u}{H^3}^2.
\end{align}

Next, we bound the rest terms involved in $\b$ in $f_2.$ Since the lower order terms are easy to deal with, we  are only concerned with the highest order-derivative terms for sake of brevity. By Leibniz's rule and Lemma \ref{ani}, we have
\begin{align}\label{yiming10}
&{\int_{\R^3}{{\nabla}^{2}} (\b\cdot\nabla\b-\b\nabla\b)\cdot{{\nabla}^{2}}\nabla a\,dx}\nn\\
&\quad=\sum_{\ell = 0}^2 \mathbb{C}_{2}^\ell\int{\nabla}^{2-\ell} \b \cdot \na{\nabla}^{\ell} \b\cdot {\nabla}^2 \nabla a \, dx+\sum_{\ell = 0}^2 \mathbb{C}_{2}^\ell\int{\nabla}^{2-\ell} \b  \na{\nabla}^{\ell} \b\cdot {\nabla}^2 \nabla a \, dx\nn\\
&\quad\le C\sum_{\ell = 0}^2\|{\nabla}^{2-\ell} \b\|_{L^2}^{\f 14}\|\p_1{\nabla}^{2-\ell} \b\|_{L^2}^{\f 14} \|\p_3{\nabla}^{2-\ell} \b\|_{L^2}^{\f 14} \|\p_1\p_3{\nabla}^{2-\ell} \b\|_{L^2}^{\f 14}\notag\\
&\qquad\qquad\qquad \times \|\na{\nabla}^{\ell} \b\|_{L^2}^{\f 12} \|\p_2\na{\nabla}^{\ell} \b\|_{L^2}^{\f12}\|{\nabla}^2 \nabla a\|_{L^2}\nn\\
&\quad\le C\norm{\b}{H^3}\norm{\nabla_h \b}{H^3}\norm{\nabla a}{H^2}\nn\\
&\quad\le\frac1{16}\norm{\nabla a}{H^2}^2+C\norm{\b}{H^3}^2\norm{\nabla_h \b}{H^3}^2.
\end{align}
Similarly,
\begin{align}\label{yiming11}
{\int_{\R^3}{{\nabla}^{2}} (I(a)\b\cdot\nabla\b)\cdot{{\nabla}^{2}}\nabla a\,dx}
=&\sum_{\ell =0,1}  \mathbb{C}_{2}^\ell\int{\nabla}^{2-\ell}I(a) {\nabla}^{\ell} (\b \cdot \na\b)\cdot {\nabla}^2 \nabla a \, dx\nn\\
&+{\int_{\R^3}I(a){{\nabla}^{2}} (\b\cdot\nabla\b)\cdot{{\nabla}^{2}}\nabla a\,dx}.
\end{align}
With an aid of the H\"older inequality, the first  term on the right hand side of \eqref{yiming11} can be bounded by
\begin{align}\label{yiming12}
\sum_{\ell = 0,1}  \mathbb{C}_{2}^\ell\int{\nabla}^{2-\ell}I(a) {\nabla}^{\ell} (\b \cdot \na\b)\cdot {\nabla}^2 \nabla a \, dx
\le&\sum_{\ell = 0,1} \|{\nabla}^{2-\ell}I(a) \|_{L^3}\| {\nabla}^{\ell} (\b \cdot \na\b)\|_{L^6}\| {\nabla}^2 \nabla a\|_{L^2}\nn\\
\le&\| \b \cdot \na\b\|_{H^2}\|  \nabla a\|_{H^2}^2\nn\\
\le&\| \b \|_{H^3}^2\|  \nabla a\|_{H^2}^2.
\end{align}
For the second term, it follows from \eqref{yiming10} that
\begin{align}\label{yiming14-1}
{\int_{\R^3}I(a){{\nabla}^{2}} (\b\cdot\nabla\b)\cdot{{\nabla}^{2}}\nabla a\,dx}
\le&
  C\|I(a)\|_{L^\infty}\norm{\b}{H^3}\norm{\nabla_h \b}{H^3}\norm{\nabla a}{H^2}\nn\\
  \le&
  C\|a\|_{H^3}\norm{\b}{H^3}\norm{\nabla_h \b}{H^3}\norm{\nabla a}{H^2}\nn\\
\le& \frac1{16}\norm{\nabla a}{H^2}^2+C\norm{(a,\b)}{H^3}^2\norm{\nabla_h \b}{H^3}^2.
\end{align}
The term ${\int_{\R^3}{{\nabla}^{2}} (I(a)\b\nabla\b)\cdot{{\nabla}^{2}}\nabla a\,dx}$ can be estimated similarly, thus, we have
\begin{align}\label{yiming14}
{\int_{\R^3}{{\nabla}^{k}} (I(a)(\b\cdot\nabla\b-\b\nabla\b))\cdot{{\nabla}^{k}}\nabla a\,dx}
\le& \frac1{8}\norm{\nabla a}{H^2}^2+C\norm{(a,\b)}{H^3}^2\norm{\nabla_h \b}{H^3}^2.
\end{align}
Making use of \eqref{yiming15} through \eqref{yiming14} leads to
\begin{align}\label{yiming18}
M_3\le&\frac12\norm{ \nabla a}{H^{2}}^2+C\norm{(a,\b)}{H^3}^2\norm{\nabla a}{H^2}^2
+C\norm{(a,\u)}{H^3}^2\norm{\nabla \u}{H^3}^2 +C\norm{(a,\b)}{H^3}^2\norm{\nabla_h \b}{H^3}^2.
\end{align}
Inserting \eqref{yiming6}, \eqref{yiming7}, and \eqref{yiming18} into  \eqref{yiming2},  we can arrive at
\eqref{yiming1} by summing up $k=0,1,2$. This completes
the proof of Proposition \ref{adis}.
\end{proof}

\subsection{Completing the proof of Theorem \ref{dingli1}} \label{compl}

Multiplying by a suitable large constant $A>1$ on both hand side of \eqref{highenergy} and then summing up \eqref{yiming1}, we can get
\begin{align}\label{finalenergy}
&\frac12\frac{d}{dt}\Big({A}\norm{(a,\u,\b)}{H^3}^2+\sum_{k=0}^2{\int_{\R^3}2{{\nabla}^{k}} \u\cdot{{\nabla}^{k}}\nabla a\,dx}\Big)\nn\\
&\qquad+\norm{\nabla a}{H^{{2}}}^2+{A}\mu\norm{\nabla\u}{H^3}^2+{A}(\lambda+\mu)\norm{\div\u}{H^3}^2
+{A}\sigma\norm{\nabla_h\b}{H^3}^2 \nn\\
&\quad\le C(\| \u\|_{H^3}+\|(a,\u,\b)\|_{H^3}^2+\| \b\|_{H^3}^4)\| \nabla a\|_{H^2}^2
\nn\\
&\qquad+C(\|  a\|_{H^3}^2\|  \b\|_{H^3}^2+\|(a,\u,\b)\|_{H^3}^2)\| \nabla \u\|_{H^3}^2+C\norm{(a,\b)}{H^3}^2\norm{\nabla_h \b}{H^3}^2.
\end{align}
We  denote the initial  energy
\begin{align*}%\label{}
\mathcal E(0)\stackrel{\mathrm{def}}{=}\norm{(a_0,\u_0,\b_0)}{H^3}^2,
\end{align*}
 and the total energy
\begin{align}\label{energy}
\mathcal E(t)\stackrel{\mathrm{def}}{=}\mathcal E_1(t)+\mathcal E_2(t),
\end{align}
with
\begin{align*}%\label{}
 \mathcal E_1(t)\stackrel{\mathrm{def}}{=}&\norm{(a,\u,\b)}{L^\infty_{t}(H^3)}^2,\nonumber\\
\mathcal E_2(t)\stackrel{\mathrm{def}}{=}&\norm{\nabla a}{L^1_{t}(H^{{2}})}^2+\mu\norm{\nabla\u}{L^1_{t}(H^3)}^2+(\lambda+\mu)\norm{\div\u}{L^1_{t}(H^3)}^2
+\sigma\norm{\nabla_h\b}{L^1_{t}(H^3)}^2 .
\end{align*}
Due to
\begin{align}\label{}
\Big|\sum_{k=0}^2{\int_{\R^3}2{{\nabla}^{k}} \u\cdot{{\nabla}^{k}}\nabla a\,dx}\Big|
\le\norm{a}{H^3}\norm{\u}{H^3}\le C(\norm{a}{H^3}^2+\norm{\u}{H^3}^2),
\end{align}
then, we can get  from \eqref{finalenergy} that
\begin{align}\label{energy3}
\mathcal E(t)\le& C_1\mathcal E(0)+C_1(\mathcal E_1(t))^{\frac12}\mathcal E_2(t))+C_1\mathcal E_1(t)\mathcal E_2(t))+C_1(\mathcal E_1(t))^{2}\mathcal E_2(t))\nn\\
\le& C_1\mathcal E(0)+C_2(\mathcal E(t))^{\frac32}+C_2\mathcal E^3(t)
\end{align}
for some  absolute   positive constants $C_1$ and $C_2$.

Under the setting of initial data in Theorem\ref{dingli1},  there exists a positive constant $C_3$ such that
$ \mathcal E (0) + C_2 \mathcal E(0)\leq C_3 \varepsilon$. Due to the local existence result which can be achieved by  a standard processes, there exists a positive time $T$ such that
\begin{equation}\label{re}
 \mathcal E (t) \leq 2 C_3\varepsilon , \quad  \forall \; t \in [0, T].
\end{equation}
Let $T^{*}$ be the largest possible time of $T$ for what \eqref{re} holds. Now, we only need to show $T^{*} = \infty$.  By the estimate of total energy \eqref{energy3}, we can use
 a standard continuation argument to show that $T^{*} = \infty$ provided that $\varepsilon$ is small enough.  We omit the details here. Hence, we finish the proof of Theorem \ref{dingli1}. $\Box$

\bigskip
\section*{Acknowledgments}
Wu was partially supported by the National Science Foundation of the United
  States under DMS 2104682 and DMS 2309748.
Zhai was partially supported by the Guangdong Provincial Natural Science
Foundation under grant 2024A1515030115.

\bigskip
\noindent{Conflict of Interest:} The authors declare that they have no conflict of interest.

\bigskip
\noindent{Data availability statement:}
 Data sharing not applicable to this article as no datasets were generated or analysed during the
current study.


\vskip .3in
\begin{thebibliography}{10}
	
	\bibitem{abidi} H. Abidi, P.  Zhang,  On the global solution of a 3-D MHD system with initial data near equilibrium,
	{\it Comm. Pure Appl. Math.}, {\bf70} (2017), 1509--1561.
	


\bibitem{Bis} D. Biskamp, {\it Nonlinear Magnetohydrodynamics}, Cambridge University Press, Cambridge, 1993.


\bibitem{Buc1}   T. Buckmaster, S. Shkoller and  V. Vicol,
   Formation of shocks for 2D isentropic compressible Euler,
   {\it Commun. Pure Appl. Math. \bf 75} (2022), 2069--2120.

  \bibitem{Buc2}   T. Buckmaster, S. Shkoller and  V. Vicol,
   Shock formation and vorticity creation for 3d Euler,
      {\it Commun. Pure Appl. Math.}, doi.org/10.1002/cpa.22067, In
   press, 2022.



  \bibitem{Buc3}   T. Buckmaster, S. Shkoller and  V. Vicol,
   Formation and development of singularities for the compressible Euler
   equations, EMS Press. DOI 10.4171/ICM2022/210. Proceedings of the
   International Congress of Mathematicians 2022.

	\bibitem{chemin}
	J. Chemin, D.S. McCormick, J.C. Robinson,  J.L. Rodrigo,
	\newblock Local existence for the non-resistive MHD equations in Besov spaces,
	\newblock {\it  Adv. Math.}, {\bf 286} (2016), 1--31.
	\newblock $\,$
	
	\bibitem{ChenWang}G.-Q. Chen, D. Wang, 	Existence and continuous dependence of large solutions for the magnetohydrodynamic equations, {\it Z. Angew. Math. Phys.,  \bf 54} (2003), 608--632.
	
	\bibitem{zhangzhifei}
	W. Chen, Z. Zhang, J. Zhou,
	\newblock Global well-posedness for the 3-D MHD equations
	with partial diffusion in periodic domain,
	\newblock {\it Sci China Math},   {\bf65} (2022),  309--318.
	\newblock $\,$
	



\bibitem{Chr1} D. Christodoulou, The formation of shocks in 3-dimensional
   fluids, EMS Monographs in Mathematics, European Mathematical
   Society (EMS), Zurich, 2007.

   \bibitem{Chr2} D. Christodoulou, The shock development problem, EMS
   Monographs in Mathematics, European Mathematical Society
   (EMS), Zurich, 2019.

\bibitem{Davi} P.A. Davidson, {\it An Introduction to Magnetohydrodynamics},  Cambridge University Press, Cambridge, England, 2001.

	\bibitem{DV05} L. Desvillettes, C. Villani,  On the trend to global equilibrium for spatially inhomogeneous kinetic systems: the Boltzmann equation, {\it Invent. Math.}, {\bf159} (2005), 245--316.
	
	\bibitem{zhaixiaoping2021arxiv}
	B. Dong, J. Wu, X. Zhai, Global small solutions to a special $2\frac12$-D compressible viscous non-resistive MHD system,
	\newblock {\it J. Nonlinear Sci.}, {\bf 33}, (2023), no. 1, Paper No. 21, 37 pp.
	
	\bibitem{zhaixiaopingnon}
	B. Dong, J. Wu, X. Zhai, Stability and exponential decay for the compressible viscous non-resistive MHD system,
	\newblock {\it Nonlinearity}, {\bf 37}, (2023), no. 7, Paper No. 075012, 27 pp.

	
	
	\bibitem{fefferman1}
	C.L. Fefferman, D.S. McCormick, J.C. Robinson, J.L. Rodrigo,
	\newblock Higher order commutator estimates and local existence for the non-resistive MHD equations and related models,
	\newblock {\it J. Funct. Anal.}, {\bf 267} (2014), 1035--1056.
	\newblock $\,$
	
	\bibitem{fefferman2}
	C.L. Fefferman, D.S. McCormick, J.C. Robinson, J.L. Rodrigo,
	\newblock Local existence
	for the non-resistive MHD equations in nearly optimal Sobolev spaces,
	\newblock {\it Arch. Ration. Mech. Anal.}, {\bf 233} (2017), 677--691.
	\newblock $\,$

	\bibitem{F04}E. Feireisl, Dynamics of Viscous Compressible Fluids, Oxford University Press, Oxford, 2004.


\bibitem{G-2019} L. Grafakos, ,\ Oh, S., {The Kato-Ponce
    inequality}, { \it Comm. Partial Differential Equations}, \textbf{39}
  (2019), 1128--1157.


	\bibitem{hoff1}
	D. Hoff, Global solutions of the Navier-Stokes equations for multidimensional compressible flow with discontinuous initial data, { \it J. Differential Equations},  {\bf120} (1995), 215--254.
	
\bibitem{hongguangyi}
G. Hong, X. Hou, H. Peng, C. Zhu,  Global existence for a class of large solutions to three-dimensional compressible magnetohydrodynamic equations with vacuum,  { \it SIAM J. Math. Anal.},   {\bf49}  (2017), 2409--2441.





%	\bibitem{Hu} X. Hu, Global existence for two dimensional compressible magnetohydrodynamic flows with zero magnetic diffusivity, arXiv: 1405.0274v1.
	
	\bibitem{HuWang} X. Hu, D. Wang, Global existence and large-time behavior of solutions to the three-dimensional equations of compressible Magnetohydrodynamic flows, {\it Arch. Ration. Mech.  Anal., \bf 197} (2010), 203--238.
	
	
	\bibitem{jiangfei2019}
	F. Jiang, S. Jiang, Nonlinear stability and instability in the Rayleigh-Taylor problem of stratified compressible MHD fluids,
	{\it Calc. Var. Partial Differ. Equ.}, {\bf58} (2019), 29.
	
	\bibitem{JZW} S. Jiang, J.  Zhang,  On the non-resistive limit and the magnetic boundary-layer for one-dimensional compressible magnetohydrodynamics, {\it Nonlinearity,} {\bf 30} (2017), 3587--3612.
	
	


\bibitem{Kato-Ponce-1988} T. Kato, G. Ponce,  {Commutator
    estimates and the Euler and Navier--Stokes equations}, {\it Comm. Pure
  Appl. Math.}, \textbf{41} (1988),  891--907.

	
	
	\bibitem{Kawashima}
	S. Kawashima, System of a Hyperbolic-Parabolic Composite Type, with Applications to the Equations of
	Magnetohydrodynamics, Ph.D. thesis, Kyoto University, 1984.

	\bibitem{Landau}
L. Landau, E. Lifshitz, Course of Theoretical Physics. Vol. 6. Pergamon Press,
Oxford, 2nd edition, 1987. Fluid Mechanics; Translated from the third Russian edition
by J. B. Sykes and W. H. Reid


	\bibitem {Xin4}
	H. Li, Y. Wang, Z. Xin, Non-existence of classical solutions with finite energy to the Cauchy problem of the compressible Navier-Stokes equations, {\it Arch. Ration. Mech. Anal.}, {\bf232} (2019), 557--590.
	
\bibitem {lihailiang}
H.  Li, X. Xu,  J. Zhang, Global classical solutions to 3D compressible magnetohydrodynamic equations with large oscillations and vacuum, {\it SIAM J. Math. Anal.}, {\bf45} (2013), 1356--1387.

	\bibitem{lijinlu}
	J. Li, W. Tan, Z. Yin,
	\newblock Local existence and uniqueness for the non-resistive MHD equations in homogeneous Besov spaces,
	\newblock {\it Adv. Math.}, {\bf 317} (2017), 786--798.
	\newblock $\,$
	
	
	\bibitem{LS1D} Y. Li, Y. Sun,  Global weak solutions and long time behavior for 1D compressible MHD equations without resistivity, {\it J. Math. Phys.}, {\bf 60} (2019), 071511, 22 pp.
	
	\bibitem{LS2D}Y. Li, Y. Sun,  Global weak solutions to a two-dimensional compressible MHD equations of viscous non-resistive fluids,
	{\it J. Differential Equations},  {\bf267}  (2019), 3827--3851.
	
	\bibitem{LXZ} F. Lin, L.  Xu, P. Zhang,  Global small solutions of 2-D incompressible MHD system, {\it J. Differential Equations}, {\bf 259} (2015), 5440--5485.
	
	
	
	\bibitem{LiZh1}Y. Liu,  T. Zhang, {Global weak solutions to a 2D compressible non-resistivity MHD system with non-monotone pressure law and nonconstant viscosity}, {\it J. Math. Anal. Appl.}, {\bf 502} (2021), Paper No. 125244, 38 pp.
	

\bibitem{Luk} J. Luk and J. Speck, Shock formation in solutions to the 2D
      compressible Euler equations in the presence of non-zero vorticity,
   {\it Invent. Math. \bf 214} (2018), 1--169.


   \bibitem{Mer}  F. Merle, P. Raphael, I. Rodnianski and J. Szeftel, On the
     implosion of a  three dimensional compressible fluid, (2019),
     arXiv:1912.11009.


	\bibitem{Nirenberg-1959} L. Nirenberg, {On elliptic partial
    differential equations}, {\it Ann. Scuola Norm. Sup. Pisa Cl. Sci.}, {\bf
  13} (1959), 115--162.
	
	\bibitem{PZZu1} R. Pan, Y. Zhou, Y. Zhu, Global classical solutions of three dimensional viscous MHD system without magnetic diffusion on periodic boxes, {\it Arch. Ration. Mech. Anal.}, {\bf 227} (2018), 637--662.
	
	
	\bibitem{Pri} E. Priest and T. Forbes, {\it Magnetic Reconnection, MHD Theory and Applications}, Cambridge University Press, Cambridge, 2000.

	
	
	\bibitem{RWZ} X. Ren, J. Wu, Z. Xiang, Z. Zhang, Global existence and decay of smooth solution for the 2-D MHD equations without magnetic diffusion, {\it J. Funct. Anal.},  {\bf 267} (2014), 503--541.
	
	
	
       \bibitem{Sid} T.  Sideris, Formation of singularities in
       three-dimensional compressible
       fluids, {\it Comm. Math. Phys. \bf 101} (1985),
       475-485.

	\bibitem{TW} Z. Tan, Y. Wang,  Global well-posedness of an initial-boundary value problem for viscous non-resistive MHD systems, {\it SIAM J. Math. Anal.},  {\bf 50} (2018), 1432--1470.
	
	
	\bibitem{Tand} E. Tandberg-Hanssen, G. Emslie,  The physics of solar flares, Cambridge University Press,  Cambridge, United Kingdom, 1988.
	
	
	\bibitem{Triebel}
	H. Triebel, Theory of Function Spaces, Monogr. Math., Birkh$\ddot{\mathrm{a}}$user Verlag, Basel, Boston, 1983.
	
	
	\bibitem{Wa1} Y. Wang,  Sharp nonlinear stability criterion of viscous non-resistive MHD internal waves in 3D, {\it Arch. Ration. Mech. Anal.}, {\bf 231} (2019), 1675--1743.
	
\bibitem{wuguochun}
G. Wu, Y. Zhang,  W. Zou, Optimal time-decay rates for the 3D compressible magnetohydrodynamic flows with discontinuous initial data and large oscillations, {\it J. Lond. Math. Soc.}, {\bf103},  (2021),  817--845.



	\bibitem{wujihong1}
	J. Wu, The 2D magnetohydrodynamic equations with partial or fractional dissipation, in: Lectures
	on the Analysis of Nonlinear Partial Differential Equations, Morningside Lectures on Mathematics,
	Part 5, MLM5, International Press, Somerville, MA, 2018, pp. 283--332.
	
	
	\bibitem{WuWu}%(MR2929600)[1088/0951-7715/25/6/1735]
	\newblock J. Wu, Y. Wu,
	\newblock { Global small solutions to the compressible 2D magnetohydrodynamic system without magnetic diffusion},
	\newblock {\it Adv.  Math.}, {\bf310} (2017), 759--888.

	\bibitem{zhaixiaoping2022}
\newblock J. Wu, X. Zhai,
\newblock {Global small solutions to the 3D compressible viscous non-resistive MHD system},
\newblock {\it Math. Models Methods Appl. Sci.},  {\bf33} (2023), 2629--2656.
		
		
\bibitem{zhuyiadv}
J. Wu,  Y. Zhu, Gloabal solutions of 3D incompressible MHD system with mixed partial dissipation
and magnetic diffusion near an equilibrium, {\it Adv. Math.}, {\bf377} (2021), 107466.

	\bibitem{WuZhu} J. Wu, Y. Zhu, Global well-posedness for 2D non-resistive compressible MHD system in periodic domain,  {\it J. Funct. Anal.},  {\bf283} (2022), Paper No. 109602.
	
	\bibitem{Xiao} Y. Xiao, Z. Xin, J. Wu, Vanishing viscosity limit for
	the 3D magnetohydrodynamic system with a slip boundary condition, {\it J. Funct. Anal., \bf 257} (2009),  3375-3394.
	
	
	\bibitem{Xin} Z. Xin, { Blowup of smooth solutions to the compressible Navier-Stokes equation with compact density},
	{\it Comm. Pure Appl. Math.},  {\bf 51} (1998), 229--240.
	
	\bibitem{xin7}
	Z. Xin, W. Yan,
	{On blowup of classical solutions to the compressible Navier-Stokes equations},
	\newblock {\it  Comm. Math. Phys.},  {\bf321} (2013),  529--541.
	\newblock $\,$
	
	\bibitem{XZ} L. Xu, P. Zhang, Global small solutions to three-dimensional incompressible magnetohydrodynamical system,
	{\it SIAM J. Math. Anal.},  {\bf 47} (2015), 26--65.
	
	


\bibitem{Yin} H. Yin, Formation and construction of a shock wave for 3-D
      compressible Euler
equations with the spherical initial data, {\it
Nagoya Math. J. \bf 175} (2004), 125--164.


	
	\bibitem{ZT} T. Zhang, Global solutions to the 2D viscous, non-resistive MHD system with large background magnetic field, {\it J. Differential Equations},  {\bf 260} (2016),  5450--5480.
	
	\bibitem{zhongxin}
	X. Zhong, On local strong solutions to the 2D Cauchy problem of the compressible non-resistive magnetohydrodynamic equations with vacuum, { \it J. Dynam. Differential Equations},  {\bf32}  (2020),  505--526.
	
\bibitem{zhuyijmp}
Y. Zhou, Y. Zhu, Global classical solutions of 2D MHD system with only magnetic
diffusion on periodic domain, { \it J. Math. Phys.}, {\bf59} (2018), 081505.
\end{thebibliography}
\end{document}